\documentclass[a4paper,notitlepage]{article}
\usepackage[latin1]{inputenc}
\usepackage[T1]{fontenc}
\usepackage{amsmath}
\usepackage{amsfonts}
\usepackage{amssymb}
\usepackage{latexsym}
\usepackage{stmaryrd}
\usepackage[english]{babel}
\usepackage[dvips]{graphicx}
\usepackage{longtable}
\usepackage{amsthm}
\usepackage{dsfont}

\usepackage{geometry}

\newtheorem{theo}{Theorem}[section]
\newtheorem{prop}{Proposition}[section]

\newtheorem{lem}{Lemma}[section]
\newcommand{\dx}{\, \text{\normalfont d}}
\newcommand{\R}{\mathbb{R}}
\newcommand{\N}{\mathbb{N}}

\newcommand{\eps}{\varepsilon}
\newcommand{\leb}{\mathcal{L}}

\renewcommand{\phi}{\varphi}
\renewcommand{\div}{\operatorname{div}}
\newcommand{\into}{\hookrightarrow}
\newcommand{\un}{\mathds{1}}

\title{Some results on Sobolev spaces with respect to a measure and applications to a new transport problem}
\author{Jean Louet\thanks{D\'epartement de Math\'ematiques, B\^at. 425,
    Facult\'e des Sciences, Universit\'e Paris-Sud 11, F-91405 Orsay
    cedex, France (\texttt{jean.louet@math.u-psud.fr})}}
\date{February 19, 2013}

\begin{document}
\maketitle
\date

\begin{abstract} We recall some known and present several new results about Sobolev spaces defined with respect to a measure $\mu$, in particular a precise pointwise description of the tangent space to $\mu$ in dimension~1. This allows to obtain an interesting, original compactness result which stays open in $\R^d$, $d>1$, and can be applied to a new transport problem, with gradient penalization. \end{abstract}

\section*{Introduction}

Let us consider variational problems, consisting in the minimization of
$$J \,:\, u \mapsto \int_\Omega L(x,u(x),\nabla u(x)) \dx\mu(x) $$
where the usual Lebesgue measure is replaced by a generic Borel measure~$\mu$, under possible $\mu$-a.e.\@ or boundary constraints. In calculus of variations, the direct method consists in extracting a converging subsequence (in a suitable sense) from a minimizing sequence, thanks to a compactness result on the set of admissible functions, and to conclude by semi-continuity of the functional $J$. For our functional, two problems appear:
\begin{itemize}
\item which functional space should we consider in order to give a sense to the gradient $\nabla u$? More precisely, if $\mu$ is the Lebesgue measure or has a density $f$ bounded from above and below, we can work in the classical Sobolev space $H^1(\Omega)$ (which is exactly the set of functions $u \in L^2_\mu$ having weak derivatives in $L^2_\mu$), but it is not so clear if this assumption on $f$ does not hold or if $\mu$ has a singular part.
\item does there exist a compactness result which allows to extract from a minimizing sequence a subsequence converging, in a suitable sense, to an admissible function? For instance, in the classical Sobolev space $H^1(\Omega)$, the Rellich theorem allows to extract from any bounded sequence a strongly-convergent subsequence in $L^2(\Omega)$ which is a.e.\@ convergent on $\Omega$.
\end{itemize}

Let us fix more precisely the notations. Let $\Omega$ be a bounded open set of $\R^d$ and $f$ a measurable and a.e.\@~positive function on $\Omega$. If we assume $f$ to be bounded from below and above, it is obvious that the set
$$ \{u \in L^2_\mu(\Omega) : \nabla u \text{ exists in the weak sense and belongs to } L^2_\mu(\Omega)^d \} $$
is exactly the classical Sobolev space $H^1(\Omega)$, since the $L^2_\mu$-norm is equivalent to the usual $L^2$-norm on $\Omega$. If $f$ is only assumed to be positive, for $u \in L^2(\Omega)$, the Cauchy-Schwarz inequality gives
$$ \int_\Omega |u(x)| \dx x = \int_\Omega \left(|u(x)| \sqrt{f(x)}\right) \frac{\dx x}{\sqrt{f(x)}} \leq \left( \int_\Omega |u(x)|^2 f(x) \dx x \right)^{1/2} \, \left( \int_\Omega \frac{\dx x}{f(x)} \right)^{1/2} $$
thus, under the assumption
\begin{equation} 1/f \in L^1(\Omega) \label{f} \end{equation}
we have the continuous embedding
$$ L^2_f(\Omega) := \left\{u : \int_{\Omega} |u(x)|^2 f(x) \dx x < +\infty \right\} \into L^1(\Omega). $$
In this case, any function $f \in L^2_f(\Omega)$ has a gradient $\nabla u$ in the weak sense, since it is locally integrable on~$\Omega$, and we can define the weighted Sobolev space with respect to $f$
$$H^1_f(\Omega) = \left\{ u \in L^2_f(\Omega) \; : \; \nabla u \in L^2_f(\Omega) \right\}. $$ 
More generally, if $p \in \, ]1,+\infty[ \, $, the assumption
$$ (1/f)^{1/(p-1)} \in L^1_{loc}(\Omega) $$
is a well-known sufficient condition to define the weighted Sobolev space $W^{1,p}(\Omega)$ with respect to $f$ (see~\cite{kufner} for more details). For our problem, if the Lagrangian functional is quadratic with respect to the gradient, for instance
$$J(u) = \int_\Omega |\nabla u|^2 \dx\mu, $$
it means that the set of admissible functions is well-defined if $\mu$ has a density $f$ such that $1/f$ is integrable: it is the set of the elements of the weighted Sobolev space $H^1_f$ satisfying the constraints.

If $\mu$ is absolutely continuous with density $f$, the weighted Lebesgue space with respect to $f$ is exactly the space $L^p_\mu$, so that the space $H^1_f$ can be seen as a Sobolev space with respect to the measure $\mu$. A natural generalization consists in defining the Sobolev space with respect to the measure $\mu$, without condition on its density or when $\mu$ is not anymore assumed to be absolutely continuous with respect to~$\leb^d$. There exists some general definitions of the Sobolev space in a generic metric measure space $(X,d,\mu)$ (see \cite{haj2}), but we will not enter to the details of this notions in this paper and we prefer to focus on the case of an open set of $\R^d$.

We begin this paper by an overview of the definitions and already known results about this Sobolev spaces \cite{bbs, fra, f-m, preiss, zhikov2, zhikov}, and present several new results: in particular, we give a precise description of the tangent space to any measure $\mu$ on the real line. As a corollary of this result, we show a compactness result in $H^1_\mu$, which states precisely that any bounded sequence admits a pointwise $\mu$-a.e.\@ convergent subsequence on the set of points where the tangent space is not null (this result is already known in any dimension under strong conditions on the measure $\mu$, when the compact embedding of the Sobolev space $W^{1,p}_\mu$ with respect to $\mu$ into the Lebesgue space $L^p_\mu$ still holds; see \cite{bf, h-k}).

This is applied to a variational problem coming from optimal transportation: we consider the minimization of the functional
$$ J(T) = \int_\Omega L(x,T(x),D_\mu T(x)) \dx\mu(x) $$
among all the maps $T:\Omega \mapsto \R^d$ which admit a Jacobian matrix $D_\mu T$ with respect to $\mu$ and under a constraint on the image measure $T_\# \mu$ (it corresponds to the classical Monge-Kantorovich optimal transportation problem \cite{vil} if $L$ does not depend on its third variable, and is linked to minimization problems under volume-preservation or area-preservation constraints \cite{tann}). In the one-dimensional case, we get the existence of a solution for any measure $\mu$ (the optimal map is known if $\mu$ is assumed to be the uniform measure on the interval, see \cite{ls} for details). However, we are not able to give a precise description of the tangent space and to obtain the existence of solution to this transport problem in the most general case in any dimension.

\section{Sobolev spaces with respect to a measure}

This section is devoted to an overview of the definitions and already known results about tangent spaces to a generic Borel measure $\mu$ and Sobolev spaces associated to this measure. First, let us recall that there exist some notions of Sobolev spaces in arbitrary metric measure spaces $(X,d,\mu)$, for instance in the papers by Shanmugalingam \cite{sha}, Haj\l asz \cite{haj} or Haj\l asz and Koskela \cite{h-k} (see \cite{haj2} for a global summary of this notions). In our case, a usual method consists in defining the tangent space to $\mu$ (which is a function defined $\mu$-a.e.\@ on $\R^d$ and taking values in the set of linear subspaces of $\R^d$), and the gradient with respect to $\mu$ for a regular function $u$ through
$$ \nabla_\mu u (x) = p_{T_\mu(x)} (\nabla u(x)) \quad \text{for }\mu\text{-a.e.\@ } x \in \R^d, $$
where $p_{T_\mu(x)}$ is the orthogonal projection on $T_\mu(x)$ in $\R^d$. Then we consider for the Sobolev space $H^1_\mu$ the closure of $C^\infty\left(\overline{\Omega}\right)$ for the norm
$$u \in C^\infty\left(\overline{\Omega}\right) \mapsto ||u||_{L^p_\mu} + ||\nabla_\mu u||_{L^p_\mu}. $$
There exist several ways to define the tangent space of a generic measure $\mu$. Preiss \cite{preiss} gives a method based on the idea of blow-up: a $k$-dimensional subspace $P_\mu$ is said to be an approximate tangent space of $\mu$ at $x$ if we have, for some $\theta > 0$, the following convergence in the vague topology of measure when $\rho$~goes to $0$:
$$ \mu(x+ \rho\, \cdot \,) \rightharpoonup \theta \mathcal{H}^k|_{P_\mu}. $$
In order to examine variational problems, Bouchitt\'e {\it et al}.\@ \cite{bbs} have introduced a dual-formulation of the tangent space: it is the $\mu$-ess.\@ union (see \cite{c-v} or later) $x \mapsto Q_\mu(x)$ of the tangent fields, {\it i.e.} the vector fields belonging to
$$ X^{p'}_\mu = \{ \phi \in (L^{p'}_\mu)^d : \div(\mu \phi) \in L^{p'}_\mu \},  $$ 
where the operator $\div(\mu v)$ is defined in the distributional sense. Fragal\`a and Mantegazza \cite{f-m} have noticed that, with this notation, we have the inclusion $Q_\mu(x) \subseteq P_\mu(x)$ for $\mu$-a.e.\@ of $\R^d$ (see the PhD.\@ thesis \cite{fra} for a complete overview and more details about these definitions). 

We are interested in another way to define tangent and Sobolev spaces, introduced by Zhikov \cite{zhikov2, zhikov}. Let $\Omega$ be a bounded open set of $\R^d$ and $\mu$ a finite positive measure on $\Omega$ We will say that $u \in L^2_\mu$ belongs to the space $H^1_\mu$ if it can be approximated by a sequence of regular functions whose gradients have a limit in the space $L^2_\mu$:
$$u \in H^1_\mu \; \Longleftrightarrow \; \exists (u_n)_n \in C^\infty(\overline{\Omega}), \, v \in (L^2_\mu)^d : \, 
\left\{
\begin{array}{l}
u_n \to u \\
v_n \to v
\end{array}
\right.
\; \text{for the }L^2_\mu\text{-norm}.  $$
The set of these limits $v$ is denoted by $\Gamma(u)$, and its elements are called gradients of $u$. In general, $u$ can have many gradients (see below the example of a measure supported on a segment of $\R^2$), and it is obvious that $\Gamma(u)$ is a closed affine subspace of $(L^2_\mu)^d$ with direction $\Gamma(0)$. The projection of $0$ onto this subspace (in the Hilbert space $(L^2_\mu)^d$) is thus the unique element of $\Gamma(u)$ with minimal $L^2_\mu$-norm: we call it {\it tangential gradient} of $u$ with respect to $\mu$.

\noindent {\bf Pointwise description of $\nabla_\mu u$ and tangent space to $\mu$.} We define the tangent space to $\mu$ as follows: the space $\Gamma(0)$ can be seen as the set of vector-valued functions which are pointwise orthogonal to the measure $\mu$. Let us denote by $(e_1, \dots, e_d)$ the canonical basis of $\R^d$, and set
$$\xi_i = p_{\Gamma(0)}(e_i)$$
where the projection is taken in the Hilbert space $L^2_\mu$ (here $e_i$ is seen as a constant function on $\Omega$). For $x \in \Omega$, we denote by
$$ T_\mu(x) = \left(\operatorname{Vect} (\xi_1(x), \dots, \xi_d(x))\right)^\perp $$
and call $T_\mu(x)$ (which is defined for $\mu$-a.e.\@ $x \in \Omega$) the tangent space to $\mu$ at $x$. Then, the following equivalence holds:
$$ v \in \Gamma(0) \quad \Longleftrightarrow \quad \text{for $\mu$-a.e.\@ } x \in \Omega, \; v(x) \perp T_\mu(x). $$
This result, combined to the orthogonality property of $\nabla_\mu u$ in $L^2_\mu$, implies a pointwise description of the tangential gradient:
\begin{prop} Let $u \in H^1_\mu$. Then, for $v \in \Gamma(0)$, the function
$$x \in \Omega \mapsto p_{T_\mu(x)}(v(x)) $$
is independent of the function $v$ and only depends on $u$, and we have
$$\nabla_\mu u(x) = p_{T_\mu(x)} \qquad \text{for } \mu\text{-a.e.\@ } x \in \Omega. $$
\end{prop}

\noindent {\bf Some natural examples.} We can see that the words ``{\it tangential} gradients'' are quite natural in the following cases:
\begin{itemize}
\item if $\mu$ is the Lebesgue measure $\leb^1$ concentrated on the segment $I=[0,1]\times\{0\} \times \dots \times \{0\}$, then $T_\mu$ is the line $\R \times \{0\} \times \dots \times \{0\}$ a.e.\@ on $I$ and
$$H^1_\mu = \left\{u \in L^2_\mu : \frac{\partial u}{\partial x_1} \in L^2_\mu\right\} \quad \text{and} \quad \nabla_\mu u = \left(\frac{\partial u}{\partial x_1},0,\dots,0\right); $$
\item more generally, if $\mu$ is the uniform Hausdorff measure supported on a $k$-dimensional manifold $M$, then $T_\mu$ is the tangent space to $M$ in the sense of the differential geometry.
\end{itemize}
Let us remark that, if $v$ is a tangent field as defined above, {\it i.e.}\@ the operator $\div(\phi \mu)$ is continuous for the $L^2_\mu$-norm on $\mathcal{D}(\Omega)$, we have for any sequence $(u_n)_n$ of smooth functions having $0$ for limit in $L^2_\mu$:
$$ \left|\int_\Omega \nabla u_n \cdot \phi \dx\mu \right| \leq C \, ||u_n||_{L^2_\mu} \to 0.  $$
Then, if $v \in \Gamma(0)$, we have $v \cdot \phi$ in $L^2_\mu$ for any element $\phi \in X^2_\mu$. We deduce that, with the above notations, the space $Q_\mu(x)$ is pointwise orthogonal to $F_\mu(x)$: it means that, up to a $\mu$-negligible set, we have the inclusion between tangent spaces
$$Q_\mu(x) \subseteq T_\mu(x). $$
We are not able for the moment to prove the inverse inclusion, but the equality between this linear spaces holds for all the examples that we have~studied.
 
\section{Precise description and compactness result in one dimension}

\subsection{The main results}

Let us now give a precise pointwise description of the tangent space $T_\mu(x)$ when $d=1$ and $\Omega$ is a bounded interval of $\R$ (which we denote by $I$). In this case, there are only two options for $T_\mu(x)$ which are of course $\{0\}$ and $\R$, and the definitions of the tangent space give the following characterizations:

\noindent {\bf Fact.} Let $B \subseteq I$ be a Borel set with $\mu(B) > 0$. We have the following implications:
\begin{enumerate}
\item if any $v \in \Gamma(0)$ is $\mu$-a.e.\@ null on $B$, then $T_\mu = \R$ $\mu$-a.e.\@ on $B$;
\item if, for any $u \in H^1_\mu$, there exists a gradient of $u$ which is $\mu$-a.e.\@ null on $B$, then $T_\mu = 0$ $\mu$-a.e.\@ on $B$;
\item if there exists a gradient of $0$ which is positive $\mu$-a.e.\@ on $B$, then $T_\mu$ = $0$ $\mu$-a.e.\@ on $B$.
\end{enumerate}

\noindent {\bf Notations.} We denote by:
\begin{itemize}
\item $\mu = \mu_a + \mu_s$ , where $\mu_a$ and $\mu_s$ are respectively the absolutely continuous and the singular part of $\mu$ with respect to the Lebesgue measure;
\item $A$ a Lebesgue-negligible set on which is concentrated $\mu_s$;
\item $f$ the density of $\mu_a$, and
$$ M = \left\{x \in I : \; \forall \eps > 0, \; \int_{I \cap B(x,\eps)} \frac{\dx t}{f(t)} = +\infty \right\}$$
which is a closed set of $I$ verifying $1/f \in L^1_{loc}(I\setminus M)$.
\end{itemize}
Notice that if $\mu$ is absolutely continuous with respect to $\leb^1$, the Sobolev space with respect to $\mu$ (thus, to $f$) is well-defined exactly ``outside of the set $M$''. In our case, we find an analogous result:

\begin{theo}
For $\mu$-a.e.\@ $x \in I$, the tangent space is given by
$$ T_\mu(x) =
\left\{
\begin{array}{ll}
\{0\} & \text{if } x \in M\cup A \\
\R & otherwise.
\end{array}
\right.
 $$
\end{theo}

Let us give a short comment of this result. Saying that the tangent space is $\R$ on a set $B$ means exactly that, if $u \in H^1_\mu$ is given, all the gradients of $u$ are equals on $B$. In our case, let us denote by $V = I \setminus (M \cup A)$ and $U = I \setminus M$. Notice that $U$ is an open subset of $I$ coinciding with $V$ up to the $\leb^1$-negligible set $A$. Let us fix $u \in H^1_\mu$. We will prove that the distributional derivative of $u|_U$ is well-defined, belongs to $L^2_f$ and that, if $v \in \Gamma(u)$, $u'=v$ $\mu$-a.e.\@ on $V$; therefore, $u'|_V$ is the only gradient of $u$ on the set $V$.

First, let us recall that if $u$ is an element of $L^2_\mu$, its restriction to $U$ belongs to $L^2_f(U)$, which is included into $L^1_{loc}(U)$ by definition of $M$. The weak derivative of $u|_U$ is thus well-defined. If $\phi$ is a test function with support in $U$ and $(u_n)_n$ a sequence of regular functions such that $(u_n,u'_n) \to (u,v)$ in $L^2_\mu$, testing $v-u'$ against $\phi$ gives
$$|< v-u',\phi>_{\mathcal{D}'(U),\mathcal{D}(U)}| = \lim\limits_{n \to +\infty} \left| \int_I (u_n-u) \phi'\right| \leq ||u_n-u||_{L^2_f} \left(\int_I \frac{(\phi')^2}{f} \right)^{1/2} $$
where the last inequality comes from the H\"o lder inequality, and the last term is finite since $\phi'$ is bounded and $1/f$ integrable on the support of $\phi$. This proves that $u'_n \to v$ in the sense of distributions on $U$. Then $v|_U$ is the weak derivative of $u$ on this set, and we know that $v|_U \in L^2_f(U)$.

Finally, any element $u \in H^1_\mu$ gives by restriction an element of the weighted Sobolev space $H^1_f(U)$ and, on $U$, $\nabla_\mu u$ and $u'$ are coinciding a.e.\@ for the regular part $f \, \leb^d$ of $\mu$. To summarize, we have just proved the following:

\begin{prop} We denote by $V = I \setminus (M\cup A)$. Let us recall that $A$ is Lebesgue-negligible and that $V \cup A$ is open; we still denote by $H^1_f(V)$ the weighted Sobolev space $H^1_f(V\cup A)$. Then, a measurable function $u$ belongs to the Sobolev space $H^1_\mu(I)$ if and only if the two following conditions are satisfied:
$$u \in L^2_\mu(I) \quad and \quad u|_V \in H^1_f(V) $$
and in this case, its $\mu$-Sobolev norm is given by
$$||u||_{H^1_\mu(I)}^2 = ||u||_{L^2_\mu(I)}^2 + ||u'||_{L^2_f(V)}^2 $$
where $u'$ is the weak derivative of $u|_V$.
\end{prop}

\noindent {\bf Compactness result in $H^1_\mu(I)$.} In order to examine variational problems in this Sobolev spaces, the following compactness result is useful (it is already known in the case of the Lebesgue measure, as a consequence of the Rellich theorem):

\begin{prop} Let $(u_n)_n$ be a bounded sequence of $H^1_\mu(I)$. Then there exists a subsequence $(u_{n_k})_k$ which admits a pointwise limit $u$ on $\mu$-a.e.\@ every point on which $T_\mu$ is $\R$.
\end{prop}

\begin{proof} We know that $V$ is exactly (up to a $\mu$-negligible set) the set of points where $T_\mu$ is $\R$. We still denote by $U = I \setminus M$. $U$ is an open set and we have $U = V \cup A$. We will show that $(u_n)_n$ admits a subsequence which is pointwise convergent on $\leb^1$-a.e.\@ any point of $U$: it will be enough to conclude that this subsequence is $\mu$-a.e.\@ convergent on $V$, since $\mu|_V$ is absolutely continuous with respect to the Lebesgue measure.

The sequence $(u_n)$ is bounded in $H^1_\mu(I)$, thus the sequence $(u_n|_U)_n$ is bounded in the weighted Sobolev space $H^1_f$. But since $U$ is exactly the set of points around which $1/f$ is integrable, we know that $L^2_f(U) \into L^1_{loc}(U)$; this implies that the sequence of the weak derivatives of $u_n$ (which are functions of $L^2_f(U)$) is bounded in $L^1_{loc}(U)$. Then $(u_n)_n$ is bounded in the Sobolev space $W^{1,1}_{loc}(U)$, and admits a subsequence which is strongly convergent in $L^1(K)$, for any compact subset $K$ of $I$. We can again extract a subsequence which is pointwise convergent on $\mu$-a.e.\@ point of $I$; the proof is complete.\end{proof}


\subsection{First part of the proof: the regular part, outside of the critical set}

First, let us prove that $T_\mu = \R$ outside of $M \cup A$. Using the first characterization of the tangent space, we take an element $g$ of $\Gamma(0)$ and we want to show that $g = 0$ $\mu$-a.e.\@ outside of $M \cup A$; by definition of~$A$, it is enough to show that $g = 0$ $\leb^1$-a.e.\@ on $U$. As in the above remark, taking a sequence of regular functions $u_n \to 0$ with $u'_n \to g$ and a test function $\varphi$ such that $1/f$ is integrable on the support of $\varphi$, we~obtain
$$ \left|\int_U u_n' \varphi\right| = \left|\int_U u_n \varphi'\right| \leq \int_U \left|u_n \sqrt{f}\right| \left|\frac{\varphi'}{\sqrt{f}}\right| \leq \left( \int_U u_n^2 f \right)^{\frac{1}{2}} \left( \int_U \frac{\varphi'^2}{f} \right)^{\frac{1}{2}} $$
which goes to $0$ as $n \to +\infty$. The same computation gives $\int_U u'_n \varphi \to \int_U g\varphi$. We deduce that $g = 0$~$\leb^1$-a.e.

\subsection{Second part: the singular part of the measure}

Second, we prove that $T_\mu = \{0\}$ for the singular part of $\mu$. We use the third characterization of the tangent space and build a sequence of $C^1$ functions $(u_n)_n$ such that
$$u_n \to 0 \quad \text{and} \quad u'_n \to \un_A \quad \text{in } L^2_\mu.$$
where $\un_A$ is the characteristic function of the set $A$; this will prove that $\un_A \in \Gamma(0)$ and imply the result.

For $n \in \N$, let $\Omega_n$ be an open set such that $A \subseteq \Omega_n$ and $\mu(\Omega_n \setminus A)+\leb^1(\Omega_n) \leq 1/n$. By Lusin theorem, there exists a continuous function $v_n$ with $0 \leq v_n \leq 1$ on $I$ and
$$ (\mu+\leb^1)(\{x \in I : v_n(x) \neq \un_{\Omega_n}(x) \}) \leq 1/n $$
Let us consider $u(x) = \int_a^x v_n(x) \dx x$, where $a$ is the lower bound of $I$. Then we have:
\begin{itemize}
\item for any $x \in I$, 
$$|u_n(x)| \leq \int_I(|v_n-\un_{\Omega_n}|(t)+\un_{\Omega_n}(t)) \dx t \leq \leb^1(\{v_n \neq \un_{\Omega_n}\}) + \leb^1(\Omega_n) \leq 2/n $$
thus $(u_n)_n$ goes to $0$ uniformly, and also in the space $L^2_\mu$;
\item on the other hand, since $u'_n = v_n$ coincides with $\un_A$ outside of a set $E_n$ such that $\mu(E_n) \leq 1/n$, we~have
$$ \int_I |u'_n(x)-\un_A(x)|^2 \dx\mu(x) \leq ||v_n-\un_A||_\infty^2 \, \mu(E_n) \leq 4/n $$
thus $u'_n \to \un_A$ in $L^2_\mu(I)$.
\end{itemize}
We obtain that $\un_A \in \Gamma(0)$, which guarantees that $T_\mu = 0$ on $A$.

\subsection{Third part: the critical set}

This part is more difficult. Given a function $u \in C^1\left(\overline{I}\right)$, we build a sequence $(u_n)_n$ of regular functions (say, $C^1$) such that $u_n \to u$ and $u'_n \to v$ for the $L^2_\mu$-norm, with $v=0$ on $M$. The strategy is the following:
\begin{itemize}
\item given a set $\Omega_n$ which is ``almost'' $M$, we start from a function $u_n$ which coincides with $u$ outside of $\Omega_n$ and is piecewise constant on $\Omega_n$ (so that its derivative is null on $M$);
\item then, using the fact that the discontinuity points of $u_n$ belong to $M$, we regularize $u_n$ around this points so that its derivative stays small for the $L^2_\mu$-norm.
\end{itemize}

First, we build our set $\Omega_n$:

\begin{lem} Let us denote by $(x_n)_n$ a sequence containing all the atoms of $\mu$. For $n \in \N$, there exists a set $\Omega_n$ such that:
\begin{itemize}
\item $\Omega_n = \bigcup\limits_{i=1}^{p_n} \, ]a_i,b_i[ \, $, with $b_i < a_{i+1}$ for each $i$, and $\, ]a_i,b_i[ \,  \cap M \neq \emptyset$;
\item $\Omega_n \supseteq M \setminus \{x_1,\dots,x_n\}$;
\item $\mu(\Omega_n \setminus ( M \setminus \{x_1,\dots,x_n\})) \leq 1/n.$
\end{itemize}
\end{lem}

\begin{proof} Let $U_n$ be an open set such that $M \subseteq \Omega_n$ and $\mu(U_n \setminus M) \leq 1/n$ (such a set exists since $\mu$ is regular from above); $U_n$ is a union of open intervals, and since $M$ is compact we can assume this union to be finite. We denote by $\Omega_n = U_n \setminus \{ x_1,\dots,x_n\}$. It is still a finite union of open intervals, containing $M \setminus \{ x_1, \dots, x_n\}$ and with $\mu(\Omega_n \setminus ( M \setminus \{x_1,\dots,x_n\})) \leq 1/n$. Moreover, we may assume that all these intervals contain an element of $M$: it is enough to remove from $\Omega_n$ the intervals which do not contain any element of $M$ (if after that we obtain $\Omega_n = \emptyset$, it means that $M \subseteq A$ and we already know that $T_\mu = \{0\}$ on $A$, so there is nothing to prove). \end{proof}

\par Let us thus take a sequence $(g_n)_n$ of piecewise constant functions such that $g_n \to u$ in $L^2_\mu$ (it is possible since $u$ is continuous, thus can be approximated uniformly on $I$ by a sequence of piecewise functions) and $||g_n||_\infty \leq C$, where $C$ only depends on $u$; we replace $g_n$ by $u$ outside of the set $\Omega_n$ (the new function will still be called $g_n$), so that we have now
\begin{itemize}
\item $g_n \to u$ in $L^2_\mu$;
\item $g_n$ coincides with $u$ outside of $\Omega_n$;
\item $g_n$ coincides on $\Omega_n$ with a piecewise constant function.
\end{itemize}

\par We begin by regularizing $g_n$ around the endpoints of the intervals forming $\Omega_n$. Let $\eps_n > 0$ be small enough so that:
\begin{itemize}
\item $a_i+\eps_n < b_i-\eps_n$, for each $i$ (we will set $a'_i = a_i+\eps_n$ and $b'_i = b_i-\eps_n$);
\item $]a'_i,b'_i[ \, $ contains at least an element of $M$, for each $i$;
\item on $\, ]a_i,b_i[ \, $, $g_n$ has not any discontinuity point outside $\, ]a'_i,b'_i[ \, $;
\item if we denote by $\Omega'_n$ the union of the intervals $\, ]a'_i,b'_i[ \, $, we have $\mu(\Omega_n \setminus \Omega'_n) \leq 1/n$.
\end{itemize}

\begin{lem} There exists a function $w_n$ coinciding with $g_n$ outside of $\Omega_n \setminus \Omega'_n$, and such that, on each interval $\, ]a_i,a'_i[ \, $ and $\, ]b'_i,b_i[ \, $,
\begin{itemize}
\item $w_n$ and $w'_n$ are bounded by constants depending only on $u$ and $u'$;
\item $w_n(a_i) = u(a_i)$, $w'_n(a_i) = u'(a_i)$ and $w'_n=0$ on a (small) open interval having $a'_i$ for upper bound;
\item $w_n(b_i) = u(b_i)$, $w'_n(b_i) = u'(b_i)$ and $w'_n=0$ on a (small) open interval having $b'_i$ for lower bound.
\end{itemize}
\end{lem}

\begin{proof} It is enough to replace $g_n$ on the interval $\, ]a_i,a_i+\eps_n[ \, $ by the function $x \mapsto Q(a_i+x)$ where
$$Q(t) = -\frac{u'(a_i)}{2\eps_n} t^2 + u'(a_i) t + u(a_i) $$
(so that $Q(0)=u(a_i)$, $Q'(0)=Q'(a_i)$ and $Q'(\eps_n)=0$), to scale the new function on the interval $\, ]a_i,a'_i[ \, $ by replacing it by
$$ x \mapsto \left\{ \begin{array}{ll}
w_n(a_i+2(x-a_i)) & \text{if } a_i \leq x \leq a_i+\eps_n/2 \\
w_n({a'_i}^-) & \text{otherwise}
\end{array} \right.$$
and to make a similar construction on the interval $\, ]b_i-\eps_n,b_i[$. \end{proof}

Since $g_n$ and $w_n$ are bounded uniformly in $n$ and coincide outside of the set $\Omega_n \setminus \Omega'_n$, whose measure is at most $1/n$, the sequence $(w_n)_n$ still converges to $u$ in $L^2_\mu$; moreover, we have
$$ ||w'_n||_{L^2_\mu(\Omega_n \setminus \Omega'_n)} \leq (2/n) ||u'||_\infty $$
and for any discontinuity point $y$ of $w_n$, $w_n$ is piecewise constant on a (small) neighborhood of $y$. We now have to regularize $w_n$ around its discontinuity points, which belong to $\Omega'_n$; this is possible with a small cost only if these points belong to the set $M$. For this reason we are interested by the following ``displacement'' procedure of the discontinuity points:

\begin{lem}
For any $n$, there exists a function $v_n$ such that
\begin{itemize}
\item $v_n = w_n$ outside of $\Omega'_n$;
\item $v_n$ is still piecewise constant on $\Omega'_n$;
\item any discontinuity point of $v_n$ belongs to $M$;
\item $v_n \to u$ in $L^2_\mu$.
\end{itemize}
\end{lem}

\begin{proof} We have to modify $w_n$ only on each interval $\, ]a'_i,b'_i[ \, $. On this interval, the number of discontinuity points of $w_n$ is finite; we denote these points by $a'_i \leq x_1 < \dots < x_n = b'_i$. We make the following construction:
\begin{itemize}
\item Let $m = \inf([a'_i,b'_i] \cap M)$. We define $v_n$ on the interval $[a'_i,m[ \, $ (if it is nonempty) by setting $v_n = w_n({a'_i}^+)$.
\item Then we reiterate the construction starting from $m$:
\begin{itemize}
\item if $\, ]m,b'_i[ \,  \cap M = \emptyset$, we set $v_n = w_n({b'_i}^-)$ on this interval, and we are done;
\item otherwise, let $m'=\inf(\, ]m,b'_i[ \,  \cap M)$. We have naturally $m' \geq m$. If $m \geq x_n$, then we set $v_n = w_n({b'_i}^-)$ on $\, ]m',b'_i[ \, $,$w_n$ on $[m,m'[ \, $ and we are done;
\item if $m = m' < x_n$, then we denote by $j$ the smallest index such that $x_j > m$, we set $v_n = w_n$ on $[m,x_j[ \, $ and we reiterate this construction starting from $x_j$;
\item finally, if $m < m' < x_n$, we set $v_n = w_n$ on $[m,m'[ \, $ and we reiterate this construction starting from $m'$.
\end{itemize}
\end{itemize}

With this construction, $w_n-v_n \neq 0$ only on $\Omega'_n$ and outside of the set $M$. Since $\mu(\Omega'_n \setminus M) \leq 1/n $ and $w_n, v_n$ are uniformly bounded, we get $||v_n-w_n||_{L^2_\mu} \leq C/n$, and we thus still have $v_n \to u$. Moreover, by construction, $v_n$ is still piecewise constant on the set $\Omega'_n$ and all its discontinuity points belong to $M$. \end{proof}

To finish, we have to modify $v_n$ around each discontinuity point, so that the new function $u_n$ is regular and admits a derivative which is small for the $L^2_\mu$-norm. This is possible thanks to the following result about embeddings between functional spaces:

\begin{lem} Let $J$ a bounded interval of $\R$, and $\mu$ a finite measure on $J$ with density $f > 0$. The following assertions are equivalent:
\begin{enumerate}
\item The function $1/f$ belongs to $L^1(J)$
\item The space $L^2_\mu(J)$ is continuously embedded into $L^1(J)$
\end{enumerate}
\end{lem}

\begin{proof} The direct implication is obvious and comes directly from the Cauchy-Schwarz inequality. For the converse one, let us assume that $\int_J 1/f = +\infty$ and set
$$E_n = \left\{t \in J : \frac{1}{n+1} \leq f(t) < \frac{1}{n} \right\} \quad \text{and} \quad l_n = \leb^1(E_n). $$
We know that $\sum\limits_n l_n < +\infty$ (it is the length of $J$) and
$$ \sum\limits_n nl_n = \sum\limits_n \int_J n \, \un_{\{n-1 \leq 1/f \leq n\}} \geq \sum\limits_n \int_J \frac{1}{f} \, \un_{\{n-1 \leq 1/f \leq n\}} \geq \int_J \frac{1}{f} = +\infty $$
thus $\sum\limits_n nl_n = +\infty$. We will build a function $U$ which is constant on each set $E_n$, belongs to $L^2_\mu$ and does not belong to $L^1$. If we denote by $u_n$ the value of $U$ on $E_n$, it is equivalent to find a sequence $(u_n)_n$ verifying
$$ \sum\limits_n u_n^2 \, (nl_n) < +\infty \quad \text{and} \quad \sum\limits_n |u_n| \, l_n = +\infty $$
To summarize, we want to prove the following statement: for any sequence $(l_n)_n$ of positive numbers such that $\sum\limits_n n l_n = +\infty$ and $\sum\limits_n l_n < +\infty$, there exists a sequence $(u_n)_n$ of positive numbers such that $\sum\limits_n u_n^2 (nl_n) < +\infty$ and $\sum\limits_n nu_n = +\infty$. By contraposition, it is equivalent to the following: for any sequence $(l_n)_n$ of positive numbers such that $\sum\limits_n l_n < +\infty$, if the following implication holds:
$$ \left( \sum\limits_n u_n^2 \,(n l_n) < + \infty \right) \Rightarrow \left(\sum\limits_n l_n |u_n| < + \infty \right)
 $$
then we have $\sum\limits_n n l_n < +\infty$. This result can be seen as a corollary of the Banach-Steinhaus theorem. Denoting by $\ell^2_{nl_n}$ the space of sequences $(u_n)_n$ such that $\sum\limits_n u_n^2 \,(n l_n) < + \infty$, the operator
$$ T_N : u \in \ell^2_{nl_n} \longmapsto \sum\limits_{n=0}^N l_n u_n  $$
is linear continuous with norm $\left(\sum\limits_{n=0}^N nl_n\right)^{1/2}$ and the assumption about $(l_n)_n$ is equivalent to
$$\forall u \in \ell^2_{nl_n} \quad \sup\limits_{N \in \N} |T_N(u)| < + \infty.$$
By Banach-Steinhaus theorem, we get $\sup\limits_{N \in \N} ||T_N|| < +\infty$ and $\sum\limits_{n \in \N} nl_n < +\infty$; the proof is complete.
\end{proof}

\noindent {\bf End of the proof of Theorem 2.1.} Thanks to the two last lemmas, we are now able to transform the function $w_n$ into a $C^1$ function $u_n$, which will provide us our approximation of $u$. Let us recall that $v_n$ coincides with $u$ outside $\Omega_n$, is piecewise constant on $\Omega'_n$, all its discontinuity points are located in $M$ and each of this points admits a neighborhood where $v_n'$ is null. Denoting by $y_1 < \dots < y_p$ the discontinuity points of $v_n$, we find $\eps_n$ such that, for each $j$, $v_n$ is constant on $\, ]y_j-\eps_n,y_j[ \, $ and $\, ]y_j,y_j+\eps_n[ \, $. Moreover, we have:
$$ \sum\limits_{j=1}^p \mu(\, ]y_j-\eps_n,y_j+\eps_n[ \, ) \leq 2p\eps_n + \sum\limits_{j=1}^p \mu(\{y_j\}). $$
We take $\eps_n$ small enough so that $2p \eps_n \leq 1/n$. On the other hand, since each $y_j$ does not belong to the set $\{x_1,\dots,x_n\}$ of the ``big atoms'' of $\mu$, we have
$$ \sum\limits_{j=1}^p \mu(\{y_j\}) \leq \sum\limits_{k \geq n} \mu(\{x_k\}). $$
Therefore,
$$ \mu \left( \bigcup\limits_{j=1}^p \, ]y_j-\eps_n,y_j+\eps_n[ \,  \right) \xrightarrow[n \to +\infty]{} 0. $$
On the interval $\, ]y_j-\eps_n,y_j+\eps_n[ \, $, thanks to Lemma 2.4, $L^2_\mu$ is not embedded into $L^1$, thus we can find a regular function $g_j$ such that
$$ \int_{y_j-\eps_n}^{y_j+\eps_n} g_j = v_n(y_j^+)-v_n(y_j^-) \quad \text{and} \quad \int_{y_j-\eps_n}^{y_j+\eps_n} g_j^2 \dx\mu \leq \frac{1}{nq}. $$
Then, we set
$$ u_n(x) = \left \{ \begin{array}{ll}
\tilde{v}_n(y_j-\eps_n) + \int_{y_j-\eps_n}^x g_j & \text{ if } y_j-\eps_n \leq x \leq y_j+\eps_n \\
v_n(x) & \text{ otherwise.}
\end{array}
\right. $$
This functions $u_n$ form our desired approximation of $u$:

\begin{prop} This sequence $(u_n)_n$ satisfies $u_n \to u$ and $u'_n \to v$ in the space $L^2_\mu$, where

$$v(x) =  \left\{ \begin{array}{ll} u'(x) & \text{if } x \notin M \text{ or is an atom of } \mu  \\ 0 & \text{otherwise.} \end{array} \right.$$
Consequently, $T_\mu = \{0\}$ $\mu$-a.e.\@ on $M$. \end{prop}

\begin{proof} We know that $v_n \to u$, thus $u_n \to u$ in the space $L^2_\mu$ outside of the intervals $\, ]y_j-\eps_n,y_j+\eps_n[ \, $. But since the total mass of these intervals goes to $0$ and $(u_n)_n$ is uniformly bounded, we get $u_n \to u$. For the derivative, since $u_n=u$ outside of $\Omega_n$, we have
$$ ||u'_n-v||_{L^2_\mu}^2 = ||u'_n-v||^2_{L^2_\mu(\Omega_n)}  = ||u'_n-v||^2_{L^2_\mu(\Omega_n \setminus M)} + ||u'_n-v||^2_{L^2_\mu(M \setminus \{x_1,...,x_n\})} $$
where the first term goes to 0 (since $(u_n)_n$ is uniformly bounded and $\mu(\Omega_n \setminus M)$ goes to 0); for the second one, we have $v=0$ on $M$, thus it is enough to prove that $u'_n$ goes to $0$ for the $L^2_\mu$-norm on $M \setminus \{x_1,\dots,x_n\}$; this term is bounded by
$$ ||u'_n||_{L^2_\mu(\Omega'_n \setminus \{y_1,\dots,y_p\})}^2 + \sum\limits_{j=1}^p u'_n(y_j) \mu(\{y_j\}).  $$
Since $(u'_n)_n$ is uniformly bounded, we know that the second term goes to $0$, and since $u_n$ is constant outside of the intervals $\, ]y_j-\eps_n,y_j+\eps_n[ \, $ the first one is equal to
$$\sum\limits_{j=1}^p \int_{y_j-\eps_n}^{y_j+\eps_n} g_j^2 d\mu$$
which, by definition of $g_j$, is smaller than $1/n$. This completes the proof. \end{proof}

\section{Application to a transport problem with gradient penalization}

\subsection{Problem statement, and the easiest case}

We investigate the following problem, which is somehow intermediate between optimal transportation and elasticity theory:

$$ \inf\left\{ \int_\Omega (|T(x)-x|^2 + |\nabla T(x)|^2) \dx \mu(x) \right\} \, , $$
where the infimum is taken among all maps $T:\Omega \to \R^d$ with prescribed image measure $T_\# \mu = \nu$ and admitting a Jacobian matrix $\nabla T$ in a suitable sense. Contrary to the Monge-Kantorovich optimal transport problem, if $\mu$ has a density $f$ bounded from above and below, then the existence of a solution is obvious and comes from the direct method of the calculus of variations; more precisely:

\begin{prop} Let $f:\Omega \to \R^d$ a measurable function such that $0 < c < f < C < +\infty$ for some constants $c,C>0$. Let $\nu \in \mathcal{P}(\R^d)$. We assume that there exists at least one Sobolev transport map between $\dx\mu = f \cdot \dx\leb^d$ and $\nu$. Then the problem
$$ \inf \left\{ \int_\Omega (|T(x)-x|^2 + |\nabla T(x)|^2) f(x) \dx x : T \in H^1(\Omega), \, T_\#\mu = \nu \right \} $$
admits at least one solution.
\end{prop}

\begin{proof} Let $(T_n)_n$ be a minimizing sequence. We can extract from $(T_n)_n$ a sequence having, thanks to the Rellich theorem, a strong limit $T$ in $L^2$, and we also can assume that $T_n \to T$ $\leb^1$-a.e.\@ on $\Omega$, thus $\mu$-a.e.\@ on $\Omega$. Then for any function $\phi \in C_b(\R^d)$ we have
$$ \forall n \in \N \quad \int_\Omega \phi(T_n(x)) \dx\mu(x) = \int_{\R^d} \phi(y) \dx\nu(y). $$
Thanks to the pointwise $\mu$-a.e.\@ convergence of $(T_n)_n$, we can pass to the limit in the left-hand-side of this equality, which gives
$$ \forall \phi \in C_b(\R^d) \quad \int_\Omega \phi(T(x)) \dx\mu(x) = \int_{\R^d} \phi(y) \dx\nu(y) $$
and $T$ satisfies the constraint on the image measure. Moreover, the functional that we consider is of course lower semicontinuous with respect to the weak convergence in $H^1(\Omega)$, and $T$ minimizes our problem. \end{proof}

\subsection{The general formulation, and the one-dimensional case}

If $\mu$ is a generic Borel measure, we replace the term with the jacobian matrix of $T$ by $\nabla_\mu T$, so that our problem is now written
\begin{equation} \inf \left\{ \int_\Omega (|T(x)-x|^2+|\nabla_\mu T(x)|^2) \dx\mu(x) \; : \; T \in H^1_\mu(\Omega) \right\}. \label{1} \end{equation}

The existence of solutions is not clear in general. In the case of the classical Sobolev space $H^1(\Omega)$, we have seen that the key point to prove the existence is the following: from any minimizing sequence $(T_n)_n$ we can extract a sequence which converges $\leb^d$-a.e.\@ on $\Omega$, and this is enough to obtain that the limit is admissible. This is not possible in general, since we don't have any equivalent of Rellich compactness theorem for the Sobolev spaces with respect to a generic measure $\mu$.

In the one-dimensional case, if $\mu$ is the Lebesgue measure, it is known that the monotone transport map between $\leb^1$ and $\nu$ is optimal for the problem \eqref{1} (see \cite{ls}). This result does not hold if we do not make any assumption of $\mu$, but we can get an existence result thanks to the $\mu$-a.e.\@ compactness result of the second section:

\begin{theo} In dimension 1, the problem \eqref{1} admits at least one solution. \end{theo}

\begin{proof} Let us begin by rewriting precisely the functional that we consider in this case: we know that $T_\mu = \{0\}$ on $M \cup A$ and $\R$ on $V$, so that we are now minimizing
$$ J : U \in H^1_\mu(I) \longmapsto \int_{V} ((U(x)-x)^2+U'(x)^2) f(x) \dx x + \int_{M \cup A} (U(x)-x)^2 \dx\mu(x). $$
Let $(U_n)_n$ be a minimizing sequence. On the set $V$, which is exactly the set where $T_\mu$ is~$\R$, we can extract from $(U_n)_n$ a $\mu$-a.e.\@ (which means $\leb^1$-a.e.\@ wherever $f \neq 0$) pointwise convergent subsequence, whose limit is denoted by $U$; let us remark that $U$ is the weak limit of $(U_n)_n$ (up to a subsequence) in the space $H^1_f$, and by semicontinuity, we have
$$ \int_V ((U(x)-x)^2+U'(x)^2) f(x)\dx x \leq \liminf \left(\int_V ((U_n(x)-x)^2+U_n'(x)^2) f(x)\dx x \right). $$
Moreover, let us set, for $n \in \N$, $\nu_n = (U_n)_\#(\mu|_{M \cup A})$ and $\tilde{U}_n$ the optimal transport map for the Monge-Kantorovich quadratic cost between the measures $\mu|_{M \cup A}$ and $\nu_n$. It is well-known that $\tilde{U}_n$ is the unique nondecreasing transport map between $\mu|_{M \cup A}$ and $\nu_n$; because of compactness properties of nondecreasing maps, we can assume that $(\tilde{U}_n)_n$ admits, for the $\mu$-a.e.\@ convergence, a limit $\tilde{U}$. For any $n$, thanks to the optimality of $\tilde{U}_n$, we have
$$ \int_{M \cup A} (\tilde{U}_n(x)-x)^2 \dx\mu(x) \leq \int_{M \cup A} (U_n(x)-x)^2 \dx \mu(x) $$
and by semicontinuity
$$ \int_{M \cup A} (\tilde{U}(x)-x)^2 \dx \mu(x) \leq \liminf \left(\int_{M \cup A} (U_n(x)-x)^2 \dx \mu(x) \right). $$
Thus, if we denote by
$$ T_n(x) = \left\{ \begin{array}{ll}
U_n(x) & \text{if } x \in V \\
\tilde{U}_n(x) & \text{if } x \in M \cup A
\end{array} \right.
\qquad \text{and} \qquad
T(x) = \left\{ \begin{array}{ll}
U_(x) & \text{if } x \in V \\
\tilde{U}(x) & \text{if } x \in M \cup A
\end{array} \right. $$
we have $T_n \to T$ $\mu$-a.e.\@ on $I$, and
$$ J(T) \leq \liminf \left(\int_V ((U_n(x)-x)^2+U_n'(x)^2) f(x)\dx \right) +  \liminf \left(\int_{M \cup A} (U_n(x)-x)^2 \dx \mu(x) \right) = \liminf J(U_n) $$
where $(U_n)_n$ is a minimizing sequence for $J$ on the set of $H^1_\mu$ transport maps between $\mu$ and $\nu$. Thus, it is enough to prove that $T$ satisfies the constraint on image measure to conclude. But for each $n$, by construction, $(T_n)_\# \mu = \nu$ and the $\mu$-a.e.\@ convergence allows to obtain the same for the limit~$T$; the proof is complete. \end{proof}

\noindent {\bf Remark.} This result can be generalized to any functional $J: U \mapsto \int_\Omega (L_1(x,U(x)) + L_2(\nabla_\mu U(x))) \dx\mu(x)$, where $L_1$ and $L_2$ have one of the following forms:
\begin{itemize}
\item $L_1$ is a transport cost such that the nondecreasing map is optimal for the Monge-Kantorovich problem: it is the case if $L_1(x,u) = h(|x-u|)$, where $h$ is a convex function. Let us notice that in particular the statement holds if we study the problem of minimization of the norm of the gradient among all Sobolev transport maps (this corresponds to $L_1 = 0$). Of course we need to assume that the class
$$ \left \{U \in H^1_\mu : U_\# \mu = \nu \text{ and } \int_\Omega L_1(x,U(x)) \dx\mu(x) < +\infty \right \} $$
is nonempty (to guarantee that $J \not\equiv +\infty$ on the set of admissible functions). Thanks to the quadratic structure of $H^1_\mu$, this is automatically the case if $L_1$ is the quadratic cost and there exists a Sobolev transport map.
\item $L_2$ is ``quadratic'', so that the space where we study the problem is actually the Sobolev space $H^1_\mu$. The natural cases are $L_2(\nabla_\mu U) = |\nabla_\mu U|^2$ or $|\nabla_\mu U - I_d|^2$, where $I_d$ is the identity matrix (in this last case, we can consider the functional $U \mapsto ||U-\text{id}||_{H^1_\mu}$, which is a Sobolev version of the quadratic transport problem where we minimize $||U-\text{id}||_{L^2_\mu}$).
\end{itemize}

\subsection{Difficulties and partial results in any dimension}

As we said in the second section of this paper, we don't have a precise pointwise description of the $\mu$-Sobolev space $H^1_\mu(\Omega)$ if $\Omega$ is an open set of $\R^d$, which was the key point for the compactness result. More precisely, the following results still hold in any dimension:

\begin{itemize}
\item Outside of the set
$$ M = \left\{ x \in \Omega : \forall \eps > 0, \, \int_{\Omega \cap B(x,\eps)} \frac{\dx x}{f(x)} = +\infty \right\} $$
we have $T_\mu = \R^d$, a.e.\@ for the regular part of $\mu$. The proof is identical to the one-dimensional case, based on the Cauchy-Schwarz inequality and the embedding $L^2_f(\Omega \setminus M) \into L^1_{loc}$.
\item Of course, we don't have anymore $T_\mu = 0$ for the singular part of $\mu$: for instance, if $\mu$ is uniform and supported on a segment, then $\mu$ is singular, and we know that $\operatorname{dim} T_\mu = 1$ on any point. However, the tangent space on the atoms of $\mu$ is known:

\begin{prop} If $x_0$ is an atom of $\mu$, then $T_\mu(x_0) = \{0 \}$. \end{prop}

\begin{proof} Let us prove it if $x_0 = 0$. We want to build a sequence of functions $(u_n)_n$ such that
$$ u_n \to 0 \quad \text{and} \quad \nabla u_n \to e \, \un_{\{0\}} $$
where $e$ is an arbitrary unit vector (this shows that any unit vector belongs to the space $T_\mu(0)^\perp$, thus $T_\mu(0) = \{0\}$). For this goal, let us consider a smooth cutoff function $\chi$ such that $0 \leq \chi \leq 1$ and
$$ \chi_n(x) = 1 \text{ if } 0 \leq |x| \leq 1 \text{ and } \chi_n(x) = 0 \text{ if } |x| \geq 2 $$
and we denote by $\chi_n(x) = \chi(nx)$. We then set $u_n(x) = \langle x,e \rangle \, \chi_n(x)$ and show that $(u_n)_n$ is the function that we are looking for. First, noting that $u_n(0) = 0$, that $u_n$ is null outside of $B(0,2/n)$ and that $0 \leq u_n \leq 1$ for any $n$, we have
$$ ||u_n||_{L^2_\mu} = ||u_n||_{L^2\mu(\Omega \setminus x_0)} \leq \mu(B(0,2/n))-\mu(\{0\}) $$
which goes to $0$ as $n \to +\infty$; this gives us $u_n \to 0$ in $L^2_\mu$. Second, for any~$n$, we have
$$ \nabla u_n(x) = \chi_n(x) \, e + \langle e,x \rangle \, \nabla \chi_n(x) = \chi_n(x) \, e + \langle e,x \rangle \, n \nabla\chi(nx). $$
Let us notice that $\nabla u_n(0) = e$ for any $n$, thus it is enough to prove that $||\nabla u_n||_{L^2_\mu (\Omega \setminus \{0\})} \to 0 $. But $\chi_n$ and $\nabla \chi_n$ are null outside of $B(0,2/n)$ and if $0 < |x| \leq 2/n$ we have
$$ |\nabla u_n(x)| \leq |e| |\chi(nx)| + |\langle e,x \rangle| n |\nabla \chi(x)| \leq C(1+n|x|) \leq 3C $$
where $C$ is an upper bound of $\chi$ and $\nabla \chi$. Thus, $(\nabla u_n)_n$ is uniformly bounded by a positive constant, and $\nabla u_n - e\, \un_{\{0\}}$ is supported on the set $B(0,2/n) \setminus \{0\}$ whose measure $\mu$ goes to $0$. This completes the proof. \end{proof}
\item Finally, we can prove that there exists absolutely continuous measure $\mu$ such that $T_\mu$ is neither $\{0\}$ nor $\R^n$ on any point of $\Omega$. We provide an explicit example:
\begin{prop} Let be $g : \,]0,1[ \,  \to \, ]0,+\infty[ \, $ such that $\int_Jg = +\infty$ for any open interval $J \subseteq \, ]0,1[ \, $. Let $\Omega = \, ]0,1[ \, ^2$, $f:(x,y) \in \Omega \mapsto g(x)$, and $\mu$ the measure with density $f$. Then the tangent space is the vertical line $\R \cdot e_2$ on $\mu$-a.e.\@ point of $\Omega$.\end{prop}
\begin{proof} We first show that $T_\mu(x)$ is at most one-dimensional on $\mu$-a.e.\@ $x \in \Omega$. Let be $u(x,y)=x$. Since the tangent space of the measure $g(x) \cdot \leb^1$ on $\, ]0,1[ \, $ is $\{0\}$, we can find a sequence of functions $(w_n)_n$ such that
$$ \int_{0}^1 |w_n(x)-x|^2 g(x) \dx x \to 0 \quad \text{and} \quad \int_0^1 |w'_n(x)|^2 g(x) \dx x \to 0 $$
and we denote by $u_n(x,y)=w_n(x)$. It is clear that $\nabla u_n \to 0$ and $u_n \to u$ in $L^2_\mu$. Then, $\nabla u = (1,0)$ but $\nabla_\mu u = (0,0)$ $\mu$-a.e.\@ on $\Omega$. This would impossible if $T_\mu$ was $\R^2$ on a non-negligible set of $\Omega$.

We now have to show that $\R \cdot e_2$ is included to $T_\mu(x)$ for $\mu$-a.e.\@ $x$. For this, we prove that any element $v=(v_1,v_2)$ of $\Gamma(0)$ satisfies $v_2 = 0$. Indeed, we have $v_2 = \lim \partial_2 u_n$ with $u_n \to 0$. For any test function $\phi$, integrating by parts with respect to $y$ (since the density of $\mu$ depends only on $x$)~gives
$$ \int_\Omega \partial_2 u_n \phi \dx\mu = -\int_\Omega u_n \partial_2\phi \dx\mu $$
which goes to $0$ since $u_n \to 0$ in $L^2_\mu$. This gives $\int_\Omega v_2 \phi \dx\mu = 0$ for any $\phi \in \mathcal{D}(\Omega)$, thus $v_2 = 0$, and the proof is complete. \end{proof}

This shows that we cannot hope to obtain a compactness result analogous to the one-dimensional case, where any bounded sequence in $H^1_\mu$ has a subsequence which converges on $\mu$-a.e.\@ $x$ such that $T_\mu(x) \neq \{0\}$: it is enough to take a sequence of functions $(u_n)_n$ depending only on~$x$ and non-compact for the a.e.\@ convergence.\end{itemize}

\end{document}